\documentclass[leqno]{article}
\usepackage{amsmath,amssymb,amsthm}
\usepackage{mathrsfs}
\usepackage{tikz}
\usetikzlibrary{cd}
\usepackage{enumitem}
\usepackage[final]{todonotes}
\usepackage[hidelinks]{hyperref}
\usepackage{graphicx}
\theoremstyle{definition}

\newcommand{\s}{\ T^\prime}
\newcommand{\y}{\ Y^\prime}

\theoremstyle{definition}
\newtheorem{definition}{Definition}[section]

\newtheorem{theorem}{Theorem}
\newtheorem{proposition}{Proposition}
\newtheorem*{remark}{Remark}
\newtheorem{lemma}{Lemma}
\newtheorem{example}{Example}
\newtheorem{notation}[equation]{Notation}

\newcommand{\Tsub}{T^\prime}
\newcommand{\Ysub}{Y^\prime}
\newcommand{\Xss}{X^{ss}}
\newcommand{\Mdp}{{M}^{\prime\prime}}
\newcommand{\Mp}{{M}^{\prime}}
\newcommand{\M}{M}
\newcommand{\omegat}{\omega_T}
\newcommand{\omegas}{\omega_{T^\prime}}
\newcommand{\Xsst}{X^{ss}_T}
\newcommand{\Xsss}{X^{ss}_{T^\prime}}

\newcommand{\lambdat}{\lambda_T}
\newcommand{\lambdas}{\lambda_{T^\prime}}

\newcommand{\latM}{M}

\newcommand{\C}{\mathbb{C}}

\newcommand{\Z}{\mathbb{Z}}
\newcommand{\setl}[1]{\left\{#1\right\}}
\newcommand{\sigmad}{\sigma^\vee}

\newcommand{\set}[2]{\left\{#1\,\middle|\,#2\right\}}
\newcommand{\gitquot}{\operatorname{/\!\!/}}
\newcommand{\map}[3]{#1_{#2#3}^{\prime}}
\newcommand{\sigmadu}{\sigma^\vee}

\newcommand{\Spec}{\text{Spec}}
\newcommand{\Q}{\mathbb{Q}}
\newcommand{\plD}{\mathfrak{D}}
\newcommand{\inpr}[2]{\left\langle #1,\, #2 \right\rangle}
\newcommand{\glsec}[2]{\Gamma\left( #1,\,#2 \right)}

\newcommand{\shfA}{\shf{A}}
\newcommand{\strshf}[1]{\mathscr{O}_{#1}}
\newcommand{\dual}[1]{{#1}^{\vee}}
\newcommand{\shf}[1]{\mathcal{A}}

\DeclareMathOperator{\Proj}{Proj}
\DeclareMathOperator{\Cone}{Cone}
\DeclareMathOperator{\Hom}{Hom}

\title{Downgrading of a T-variety \footnote{%
  MSC2020: 14M25; 14L30, 52B20
}}

\author{
  Pavankumar Dighe\footnote{\textit{Email:} \texttt{pavankumardighe@gmail.com}}
  \and
  Vivek Mohan Mallick\footnote{\textit{Email:} \texttt{vmallick@iiserpune.ac.in}}
}

\begin{document}
\maketitle
\begin{abstract}
For an affine $T$-variety $X$ with the action of a torus $T$, this paper provides a combinatorial description of $X$ with respect to the action of a subtorus $T' \subset T$ in terms of a $T/T'$-invariant pp-divisor. We also describe the corresponding GIT fan. 
\end{abstract}

\section{Introduction}
This paper studies normal algebraic varieties over algebraically closed fields of characteristic zero admitting
an action of an algebraic torus. Such varieties are called $T$-varieties in the literature.
The most common examples of $T$-varieties are toric varieties, where the dimension of the torus equals the dimension of the variety on which it acts.
In \cite{keknmusd:toroidalembeddingsi}, the authors studied toric varieties defined over discrete valuation rings, which, in some sense, are $T$-varieties \cite{vollmert:tvardvr}.
The study of toric varieties is facilitated by the existence of associated combinatorial data which encodes a gamut of geometric properties in terms of combinatorial properties.
In case of $T$-varieties, the associated data consists of a combination of algebraic geometric and combinatorial data.
Altmann and Hausen \cite{ah:affinetvar} described this for the affine case, and the same authors along with Suss \cite{ahs:gentvar} elaborated the picture for the general case.

As mentioned, a $T$-variety is a normal variety $X$ along with an effective action of a torus $T$. The \emph{complexity} of a $T$-variety is the number $c(X) = \dim X - \dim T$. The description of $T$-variety involves a variety of dimension $c(X)$ and some combinatorial data encoded in the form of pp-divisors, i.e.\ divisors where the coefficients come from the Grothendieck group associated with the semigroup of polyhedra having a common tail cone. These data were deduced using Geometric Invariant Theory applied to the action of the torus on the variety.
The presentment of $T$-varieties in terms of pp-divisors have been fertile and a lot of geometric properties can be translated into combinatorics.
A summary of the development till around 2012 can be found in \cite{aipsv:geomtvar}.

One useful way to construct examples of $T$-varieties is by taking a known $T$-variety, say $X$ with the action of a torus $T$ denoted, temporarily, by $T \circlearrowleft X$. Assume that $\Tsub \subset T$ is a subtorus. Then $X$ is also a $T$-variety with respect to the action of $\Tsub$, $\Tsub \circlearrowleft X$. $\Tsub \circlearrowleft X$ is called a \emph{downgrading} of $T \circlearrowleft X$.
The case when $X$ is a toric variety was already studied by Altmann and Hausen \cite[section 11]{ah:affinetvar}. Ilten and Vollmert \cite{iltenvollmert:upgrading} gave a description of downgrading, and also discussed the reverse construction of upgrading for complexity one $T$-varieties. Such a downgrade should have a description in terms of a pp-divisor which has an $T / \Tsub$ action. This was mentioned in \cite[remark 2.9]{iltenvollmert:upgrading} without proof. In the second part of the paper, we give the details of this construction.

In section \ref{prelim}, we briefly recall the language of pp-divisor and the notion of GIT-data (see \ref{gitdata}) for a $T$-variety. In section \ref{posetandgitfan}, we defined a poset [\ref{poset}] describing the generators of the GIT-fan for an affine toric variety. Using the poset, we describe the GIT-data for a toric variety and downgraded affine $T$-variety. We illustrate this by an example: Example \ref{mainexample}.
In section \ref{torusinvariantY} and \ref{torusinvariantppdivisor}, we prove that the base space $\Ysub$ is a $T$-variety and for the right choice of a section (cf \ref{sec}) there is a torus invariant pp-divisor $\plD^\prime$ on $\Ysub$ such that $X(\plD^\prime)=X$.

\subsection*{Acknowlegements}
The first author thanks UGC (UGC-Ref.No: 1213/(CSIR-UGC NET JUNE 2017) ) for funding the research
and IISER Pune for providing facilities. The second author thanks IISER Pune for providing 
an excellent environment to conduct this research.

\section{Preliminaries}\label{prelim}

\subsection{$T$-varieties}
A $T$-variety is an algebraic variety along with an effective action of an algebraic torus $T$. A complexity of $X$ is $\text{dim}(X)-\text{dim}(T)$.
Let $Y$ be a normal, semiprojective variety, and $N$ be a lattice of finite rank. A \emph{tail cone}, tail($\Delta$), of any polyhedron $\Delta \subset N_\Q$ is defined as
\begin{equation*}
  \text{tail}(\Delta) := \set{v \in N_\Q}{v + \Delta \subseteq \Delta}.
\end{equation*}

For a fixed strongly convex integral polyhedral cone $\sigma$ in $N_\Q$, the collection of polyhedra with tail cone $\sigma$ forms a semigroup under the Minkowski sum. We denote this semigroup by $\text{Pol}^+(N,\sigma)$. 
A \emph{polyhedral divisor} on $Y$ with a tail cone $\sigma$ is a formal sum
$\plD = \sum_i \Delta_i \otimes D_i$, where $D_i$ runs over all prime
divisors of $Y$, and the coefficients $\Delta_i \in \text{Pol}^+(N,\sigma) $, such that only finitely many $\Delta_i$ are different from $\sigma$.
For each $u \in \sigmad $, the \emph{evaluation} of $\plD$ at $u$ is the Weil divisor \begin{equation*}
  \plD(u) := \sum_i \biggl( \min_{v \in \Delta_i} \inpr{u}{v} \biggr) D_i.
\end{equation*}

\begin{definition}
    A \emph{proper polyhedral divisor or pp-divisor} on $(Y,N)$ is a polyhedral divisor $\plD$, such that for each $u \in \sigmad$, $\plD(u)$ is a semiample, rational Cartier divisor on $Y$, which is big whenever $u$ is in the relative interior of $\sigmad$.
\end{definition}

  \label{def:afftvarpld}
  An affine scheme associated to $\plD$ is defined as follows.
  Let, for $u \in \dual{\sigma} \cap M$, $\shfA_u = \strshf{Y}(\plD(u))$.
  Then, $\shfA = \bigoplus_{u \in \dual{\sigma} \cap M} \shfA_u$ is a sheaf
  of $\strshf{Y}$ algebras. Define, the \emph{affine scheme associated with}
  $\plD$ to be
  \begin{equation*}
    X(\plD) = \Spec \left( \bigoplus_{u \in \dual{\sigma} \cap M}
    \glsec{Y}{\shfA_u} \right) = \Spec (\glsec{Y}{\shfA}).
  \end{equation*}
By \cite[Theorem 3.1]{ah:affinetvar}, $X(\plD)$ is an affine $T$-variety of complexity $\text{dim}(Y)$. The following theorem says that every affine $T$-variety arises in this way.

\begin{theorem}{\cite[Theorem 3.4]{ah:affinetvar}} Let $X$ be a $T$-variety of a complexity $\text{dim}(X)- \text{dim}(T)$. Then, there is a pp-divisor $\plD$ on $(Y,N)$, such that $X(\plD) \cong X$ where $Y$ is a normal, semiprojective variety, and $N$ is a lattice of rank $\text{dim}(T) $.
\end{theorem}

\subsection{GIT quotients}
In this section, we recall some results regarding geometric invariant theory of $T$-varieties from \cite{ah:affinetvar} and \cite{bh:gittoric}, which we need later. Consider the following setup. Let $\latM$ be a finite rank lattice, and let $A$ be an integral, finitely generated, $\latM$-graded 
$\C $-algebra
\[ A= \bigoplus_{u \in \latM} A_u.   \]
Consider an affine variety $X=\text{Spec}(A)$ with an action of the algebraic torus $T=\Spec(\C[\latM])$, induced by the $\latM$-grading on $A$. Let $L$ be the trivial line bundle on $X$ with the following torus action: \[t\cdot(x,c)=(t \cdot x,\chi^u(t)\cdot c),\]  where $\chi^u$ is the character corresponding to  $u \in \M $.
Now consider the canonical projection $L \to X$. This map is a torus equivariant map, and is the $T$-linearization of the trivial line bundle on $X$ with respect to $\chi^u$. 
\begin{definition}[\cite{mfk:geominvth}]
A $T$-linearization of a line bundle on $X$ is a line bundle  $L \to X$ along with a fiberwise linear $T$-action on $L$ such that the projection map is a torus equivariant.

\end{definition}
\begin{remark}{\cite{MR3888690}}
Any $T$-linearization of a trivial line bundle over $X$ is the linearization corresponding to the unique character, described above. 
\end{remark} 
\begin{remark}
    
An invariant section of the linearization of a trivial line bundle with respect to $\chi^u$ is precisely an element of some $A_{nu}$ for $n>0$.
\end{remark}
\begin{definition}[semistable points] The set of semistable points associated to a linearization of a trivial line bundle is denoted by $\Xss(u)$ and defined as:

 \[\Xss(u) := \bigcup_{f \in A_{nu}, \; n \in \Z_{>0} } X_f.\] \label{semistablepoint}
\end{definition}
If two linearized line bundles have same set of semistable points, 
we say that they are \emph{GIT-equivalent}. We recall the description of the GIT-equivalence classes given by a linearization of the trivial line bundle in terms of the orbit cones. We will illustrate this by an example of an affine toric variety. The following definitions are from [BH06].

\begin{definition}Consider a point $x \in X$. The \emph{orbit monoid} associated to $x \in X$ is the submonoid $S_T(x) \subset M$ consisting of all $u \in M$ that admit an $f \in A_u$ with $f(x) \neq 0$.
    The convex cone generated by $S_T(x)$ is called the \emph{orbit cone}, denote it by $\omega_T(x)$.
   The sublattice generated by the orbit cone $\omega_T(x)$ is called the \emph{orbit lattice}, denote it by $M(x)$.

\end{definition}

\begin{definition}
     The \emph{weight cone} $\omega \subset \M$ is a cone generated by $u \in M$ with $A_u \neq \setl{0}$
\end{definition}

\begin{definition}[\emph{GIT-cone}]
The \emph{GIT-cone} associated to an element $u \in \omega \cap \latM$ is the intersection of all orbit cones containing $u$, and is denoted by $\lambda(u)$. The collection of GIT-cones forms a fan and called a \emph{GIT-fan}.
\end{definition}
For the sake of brevity, we summarized all this data associated with a torus action on an affine variety as follows.
\begin{definition}\label{gitdata}
   Suppose $X$ is a $T$-variety with the action of a torus $T$. The \emph{GIT-data} associated with $(X,T)$ consists of orbit monoids, orbit cones, orbit lattices, GIT-cones, and set of semistable points.
\end{definition}
From \cite[Proposition 2.9]{bh:gittoric}, we have an order-reversing  one-to-one correspondence between the possible sets of semistable points induced by a linearization of the trivial line bundle and GIT-cones.
Consider an affine toric variety $X=\Spec(A)$ with a character lattice $M$, and dual lattice $N$. Note : \[ A=\bigoplus_{ u \in \sigmad \cap M} \C \cdot \chi^{u}\] where $\sigma$ is the polyhedral cone in $N$, and $\sigmad$ is its dual. The cone $\sigmad$ is a full dimensional cone, and each $u \in \sigmad$ is saturated\footnote{An element $u \in M$ is saturated if $A_{(v)}=\bigoplus_{n \in \Z_{\geq 0} } A_{nv}$ generated by degree one elements.}.
Using \ref{semistablepoint}, we will compute the $\Xss(u)$ . Consider a minimal generating set  $\{u_1,u_2 \dots u_k \}$ of a cone $\sigmad \cap M$. 
For $u \in \sigmad \cap M$,
\[ \Xss(u)= \Spec(A_{\chi^{u}})  \]. If $u=\alpha_1 \cdot u_1 + \alpha_2 \cdot u_2 \dots + \alpha_k \cdot u_k$ with $\alpha_i \geq 0$ then \[ \Xss(u)=  \Spec(A_{\chi{u}})= \bigcap_{\alpha_i \neq 0 } \Spec(A_{u_i}).\]

\section{Description of GIT-fans and semistable points. }\label{posetandgitfan}
\begin{lemma}[Collection of all semistable points] Let $\{u_1,u_2 \dots u_k \}$ be a minimal generating set of a cone $\sigmad \cap M$. The collection of a semistable points is $\set{\Xss(u)}{u=\alpha_1 \cdot u_1 + \alpha_2 \cdot u_2 \dots + \alpha_k \cdot u_k, \alpha_i= 0,1}$.
\end{lemma}
\begin{proof} First observe that above collection is a finite set. If $u=\alpha_1 \cdot u_1 + \alpha_2 \cdot u_2 \dots + \alpha_k \cdot u_k$ then \[\Xss(u)=\Xss(\sum_{\alpha_i \neq 0}u_i)=\bigcap_{\alpha_i \neq 0}\Xss(u_i)\].\end{proof}

From  \cite[Proposition 2.9]{bh:gittoric}, we are going to compute GIT-cone for each $\Xss(u)$. To do this we are going to compute $\Xss(u_i)$ for each $i \in \{1,2 \dotsc , k \}$. Consider the following poset, \[\Biggl( S=\set{\sum_{i=1}^k \alpha_i \cdot u_i}{\alpha_i=0,1}, \geq \Biggl) \]
where, for $v,w \in S$, $ v \geq w$ \label{poset}if $\Xss(v) \subset \Xss(w) $.

\begin{lemma}[Collection of all GIT-cones] With the above notations continuing, for $v \in S$, GIT-cone $\lambda_T(v)$ is generated by the set $\set{u_i}{v \geq u_i}$.
    
\end{lemma}
\begin{proof}
The cone generated by set $\set{u_i}{v \geq u_i}$ is denoted by $\sigma_T(v)$. First, observe that if $\Xss(v) \subset \Xss(u_i)$, then for $x \in X$, $\chi^{v}(x) \neq 0$ if and only if $\forall \; u_i \leq v$, $\chi^{u_i}(x) \neq 0$. Hence $u_i \in \omegat(x) $ for all $x$ such that $\chi^{v}(x) \neq 0$ hence $u_i \in \lambdat(v)$, so $\sigma_T(v) \subset \lambdat(v)$. If $u \in \lambdat(v)$, then $\Xss(v) \subset \Xss(u) \subset \Xss(u_i)$ where $u_i$ is a summand of $u$ hence $v \geq u_i$ hence $u \in \sigma_T(v)$.

%for all $x \in \Xss(u)$ hence $u_i \in \omegat(x)$ for all $x$ such that $v \in \omegat(x)$ . If  $u \in \lambda_T(v)$ then $\chi^u(x) \neq 0$ for all $x \in \Xss(v)$ therefore $\Xss(v) \subset \Xss(u)$.
\end{proof}

\subsection{Description of a GIT-fan with respect to the
action of a subtorus}
Consider the following setup. Let $X$ be a normal affine variety  with an effective $T$-action. Consider a subtorus  $\s$ of the torus $T$ with canonical action on  $X$. First, we have the following exact sequence from the torus inclusion,
\[  \begin{tikzcd}
    0 \arrow[r]  & \Mdp \arrow[r] & \M \arrow[r, "i"] & \Mp \arrow[r] & 0
\end{tikzcd}, \label{exactseq} \]
where $\M$ and $\Mp$ are the character lattices corresponding to the tori torus $T$ and $\s$, respectively. The lattice $\Mdp$ is the kernel of the lattice homomorphism $i$. If $X=\Spec(A)$, then we have a grading $A=\bigoplus_{u \in \M} A_u $ with respect to $T$ and, similarly, for the $\s$ action we have an induced grading $A=\bigoplus_{v \in \Mp} A_v $. In addition we have, \[ A_v= \bigoplus_{i(u)=v} A_u .\]
 We wish to compute the GIT data associated with $(X,\Tsub)$ from the GIT data associated with $(X,T)$ and above exact sequence. 
\begin{notation}
    We are using the same notation $i$ for lattice homomorphism and vector space homomorphism. 
\end{notation}

\begin{proposition} \label{subtorusgitdata} \begin{enumerate}
    \item Let $\omegat$ and $\omegas$ be the weight cones associated with the $T$ action and the $\Tsub$ respectively, then $i(\omegat) = \omegas$.

    Consider a point $x \in X$.
    
    \item Let $\omegat(x)$ and $\omegas(x)$ be the orbit cones associated with the $T$ and $\Tsub$ action respectively, then $i(\omegat(x)) = \omegas(x)$.
    \item Let $S_T(x)$ and $S_\Tsub(x)$ be the orbit monoid associated with the $T$ and $\Tsub$ action respectively, then $i(S_T(x)) = S_\Tsub(x)$.
\end{enumerate}
\end{proposition}
\begin{proof}
    Note, $A_v=\underset{i(u)=v}{\bigoplus}A_u$ then $A_u \neq 0$ for some $u$ if and only if $A_v \neq 0$. Now the results follows from linearity of $i$ and definition of $\omegat$ (resp. $\omegas$). The statements for orbit monoids and orbit cones follows similarly
\end{proof}

 %associated with it's respective tori.
\begin{proposition}\label{semistable} 
\begin{enumerate}
    \item Let
    $\Xsst (u)$ be the semistable point associated with $u \in \M$ and
    let $\Xsss (v)$ be the semistable point associated with $v \in \Mp$, then
    \[ \Xsss(v) = \underset{i(u)=nv, \; n \in \Z_{>0} }{\bigcup} \Xsst(u). \]

%\begin{proof}
%    The set $\Xsss(v)= \set{x \in X}{f(x) \neq 0 \; \text{ for some} \; f \in A_{nv}}$, but $A_{nv}=\underset{i(u)=nv}{\bigoplus} A_u$. If $f(x) \neq 0$ for $x \in X$ and for some $f \in A_{nv}$ then there is $f_1 \in A_u$ summand of $f$ for some $u $ such that $i(u)=nv$ and $f_1(x) \neq 0$. Hence \[\Xsss(v) \subset \underset{i(u)=nv \; n \in \Z_{>0} }{\bigcup} \Xsst(u) .\]
%Similarly, one can prove $\Xsss(v) \supset \underset{i(u)=nv \; n \in \Z_{>0} }{\bigcup} \Xsst(u)$.
%\end{proof}

Because of the correspondence between the GIT-cones and sets of semistable points, we have the following result.

  \item   Let $\lambdat (u)$ be GIT-cone associated with $u \in \M$ ( under the $T$-action) and
     $\lambdas (v)$ be GIT-cone associated to $v \in \Mp$ ( under the $\Tsub$-action) then
     \[ \lambda^{\prime }(v)= \lambdas (v)= \bigcap_{i(u)=v} i(\lambdat (u)) . \]

\end{enumerate}
\end{proposition}
%\begin{proof}
 %   We have $i(\omegat(x)= \omegas(x)$, hence from the definition of GIT cone proof is evident.
%\end{proof}

\begin{example}\label{mainexample}
    Lets  take $\sigma=\Cone(\setl{e_1,e_2,e_1 + e_3,e_2 +e_3})$, then \[\C[S_\sigma]= \bigoplus_{u \in \sigmadu \cap M} \C \cdot \chi^u= \C[u,v,w,uvw^{-1}] \equiv \frac{\C[x,y,z,w]}{\langle xy-zw \rangle}  \]

where $\sigmadu=\Cone(\setl{e_1,e_2,e_3,e_1 + e_2- e_3})$.
For this example we have  GIT-fans shown in the figure \ref{ts} and semistable points correspondence,

\begin{align*}
    \Xss(e_1) & \longleftrightarrow \Cone(\setl{e_1}) \\
\Xss(e_2) & \longleftrightarrow \Cone(\setl{e_2})  \\
    \Xss(e_3) & \longleftrightarrow \Cone(\setl{e_3})  \\
     \Xss(e_1 +e_2)=\Xss(e_1+e_2+e_3+e_1+e_2 - e_3) & \longleftrightarrow \Cone(\setl{e_1 ,e_2})  \\
      \Xss(e_1 +e_2-e_3)=\Xss(e_1+e_2+e_1+e_2 - e_3) & \longleftrightarrow 
      \Cone(\setl{e_1 ,e_2,e_1 +e_2 -e_3})  \\
      \Xss(e_1+e_1+e_2 - e_3) & \longleftrightarrow \Cone(\setl{e_1,e_1+e_2 - e_3}) \\
      \Xss(e_2+e_1+e_2 - e_3) & \longleftrightarrow \Cone(\setl{e_2,e_1+e_2 - e_3}) \\
      \Xss(e_1+e_3) & \longleftrightarrow \Cone(\setl{e_1,e_3}) \\
      \Xss(e_2+e_3) & \longleftrightarrow \Cone(\setl{e_2,e_3}). 
\end{align*}

\end{example}

\begin{figure}
        \centering
        \includegraphics[width=8cm]{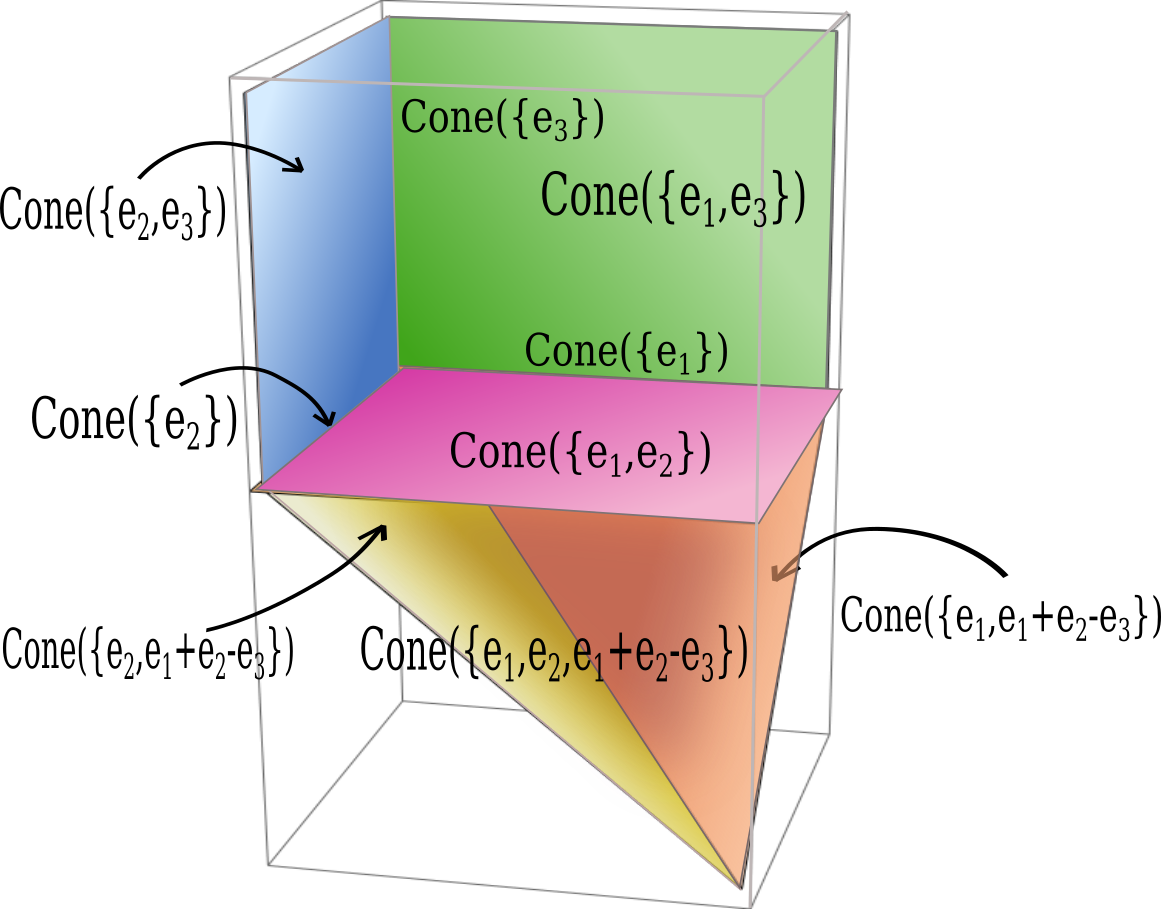}
        \caption{ $(\C^*)^3$-action}
        \label{ts}
    \end{figure}
Consider the above, the torus inclusion map $(\C^*)^2 \to (\C^*)^3 $ is given by the following map \[ (t_1,t_2) \mapsto (t_1,t_2,t_1).\] The lattice homomorphism associated with this inclusion is the map $\Z^3 \to \Z^2 $ is
\[ (a,b,c) \mapsto (a+c,b).\] 
From Proposition \ref{subtorusgitdata} and Proposition \ref{semistable}, the GIT-cones shown in the figure \ref{tdgs} and semistable points correspondence, $e_1^\prime=(1,0)$ and $e_2^\prime= (0,1)$
\begin{align*}
    \Xss(e_1^\prime)=\Xss(e_1)\cup \Xss(e_1 + e_3) \cup \Xss(e_3) \longleftrightarrow \Cone(\setl{e_1^\prime}) \\
    \Xss(e_2^\prime)= \Xss(e_2) \cup \Xss (e_1 + e_2 - e_3) \cup \Xss(e_2 + e_1 + e_2 -e_3) \longleftrightarrow \Cone(\setl{e_2^\prime})  \\
     \Xss(e_1^\prime + e_2^\prime)= \Xss(e_1 + e_2) \cup \Xss (e_1 + e_1 + e_2 - e_3) \cup \Xss( e_1 + e_3) \longleftrightarrow \Cone(\setl{e_1^\prime,e_2^\prime}). \\
\end{align*}

\begin{figure}
    \centering
    \includegraphics[width=4cm]{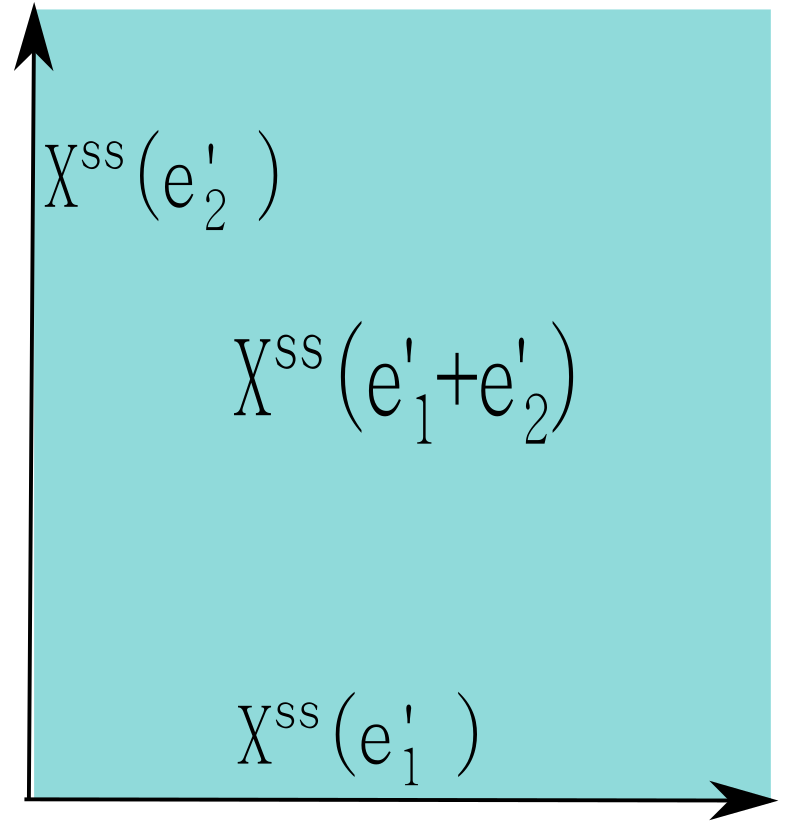}
    \caption{$(\C^*)^2$-action}
    \label{tdgs}
\end{figure}

\section{Description of the space associated to a $T$-
variety with respect to the action of a subtorus}\label{torusinvariantY}
By \cite{ah:affinetvar}, we have a proper polyhedral divisor associated with a normal affine variety with an effective torus action. Let $X$ be an affine $T$-variety, with a torus $T$. Let $\plD$ be a pp-divisor on $(Y,N)$, where $N$ is a dual lattice of a lattice $M=\Hom(T,\C^*)$, such that $X \cong X(\plD)=\Spec(A)$. We assume that $\plD$ is a minimal pp-divisor(\cite[Definition 8.7]{ah:affinetvar}). For a subtorus $\Tsub$ of the torus $T$, we shall construct $\Ysub$ and a $\frac{T}{\Tsub}$-invariant pp-divisor on $(\Ysub, N^\prime)$, where $N^\prime $ is a dual of a lattice $M^\prime= \Hom(T^\prime, \C^*)$.

\begin{theorem} \label{mainth}
Consider an affine  $T$-variety $X$ with the action of the torus $T$, with weight cone $\omega \subset M_\Q$. Let $\Tsub$ be a subtorus of $T$, and the associated lattice map be $i: M \to \Mp$. Then, there exists $\Ysub$ and a $\frac{T}{\Tsub}$-invariant pp-divisor $\plD^\prime$ on $(\Ysub,N^\prime)$ with $\text{tail}(\plD^\prime)=i(\omega)$ such that $X(\plD^\prime) \cong X$.
\end{theorem}

The next part of this paper is about the proof of the Theorem \ref{mainth}. Consider the semistable point $\Xsss(v)$, from \cite[Section 5]{ah:affinetvar}, the quotient space \\ $Y_v = \Xsss(v) \gitquot \Tsub$ is given by \[ Y_v = \Proj(A_{(v)}),\]  where $A_{(v)}=\bigoplus_{n \in \Z_{\geq 0} } A_{nv}$. %The quotients $Y_v$ are called the \emph{GIT quotients}.

\begin{proposition} \label{efp}

There is effective $\frac{T}{\s}$ action on $Y_v$ 

\end{proposition}

\begin{proof}
The set of semistable points $\Xsss(v)$ is a $\s$-invariant open subset of $X$. Moreover from Proposition \ref{semistable}, it is $T$-invariant. On basic open subsets of $Y_v$, the torus $\frac{T}{\s}$ acts effectively. Consider $f \in A_{(v)}$, in particular, we choose $f \in A_u$ for $u \in M$ such that $i(u)=nv$, for some $n \in \Z_{\geq 0}$. The sets $\Spec([A_f]_0)$ covers $Y_u$ and it is enough to prove that  $\frac{T}{\s}$ acts effectively on each $\Spec([A_f]_0)$, or equivalently, that $[A_f]_0$ ($0 \in \Mp$) admits an $\Mdp$ grading such that the weight cone is full dimension. Note that $A_0=\bigoplus_{u \in \Mdp}A_u$, which induces an $\Mdp$ grading on $[A_f]_0$ . Since $A_0=\underset{{u \in \M \; i(u)=0}}{\bigoplus}A_u\subset [A_f]_0$, and $T$ acts effectively on $X$, (the dimension of $\Cone(\set{u \in \M}{i(u)=0}$ is equal to the rank of $\Mdp$). Then, the $\frac{T}{\s}$ action is effective on $\Spec([A_0])$, and hence $\frac{T}{\s}$ acts effectively on $\Spec([A_f]_0)$.
\end{proof}
Using proposition \ref{efp}, we are going to prove that there is a $\frac{T}{\s}$ action on $\y$. From \cite[Proposition 2.9, Definition 2.8]{bh:gittoric}, the collection of all GIT-cones define the GIT-fan, which we denote by $\Sigma_{\s}$. The map $v \to \Xsss(v)$ is constant on $\text{relint}(\lambdas(v))$. So for $\lambda^\prime \in \Sigma_{\s}$ and $w \in \text{relint}(\lambda^\prime)$ if we write $W_{\lambda^{\prime}}=\Xsss(w)$, $Y_{\lambda^\prime}=Y_w$, we have the following commutative diagram
%For $\gammas \subset \lambdas$, we have $W_{\lambda} \subset W_{\gammas}$. %$W_{\lambdas}=W_{\lambda^{\prime}}$,$W_{\gammas}=W_{\gamma^{\prime}}$, %$\lambdas=\lambda^{\prime}$, $\gammas=\gamma^{\prime}$.

\begin{equation*}
\begin{tikzcd}
    W_{\s} \arrow[r, "\map{j}{\lambda^\prime}{}"] \arrow[d, "\map{q}{}{}"] & W_{\lambda^\prime} \arrow[r, "\map{j}{\lambda^\prime}{\gamma^\prime}"] \arrow[d, "\map{q}{\lambda^\prime}{}"] & W_{\gamma^\prime} \arrow[r, "\map{j}{\lambda^\prime}{0}"] \arrow[d, "\map{q}{\gamma^\prime}{}"]  & X = W_0\arrow[dddd, "\map{q}{0}{}"] \\
    \Ysub   \arrow[r, "\map{p}{\lambda^\prime}{}"] \arrow[rrrddd, "\map{p}{0}{}" ]           & Y_{\lambda^\prime} \arrow[r, "\map{p}{\lambda^\prime}{\gamma^\prime}"] \arrow[rrddd,"\map{p}{\lambda^\prime}{0}"]         & 
    Y_{\gamma^\prime} \arrow[rddd,"\map{p}{\gamma^\prime}{0}"]\\
                               &                                   & 
                        &      \\ &          &         & \\ &    &    & Y_0
\end{tikzcd} \label{maindiagram}
\end{equation*}

  All the $\map{j}{-}{-}$ are inclusion maps, so the inverse limit, $W_{\s}$, is the intersection of sets of semistable points. The
$\y$ is normalization of a canonical component. Let $Y_1$ is limit 
$\{ Y_{\lambdas} \}'s$ and \[ \Ysub= \text{Norm}(\overline{q(W_{\s})}),\label{Yistorusinvariant} \]
where Norm(-) denotes the normalization. We have the following commutative diagram

\begin{equation*}
    \begin{tikzcd}
        W_{\lambda^\prime} \arrow[r, "\map{j}{\lambda^\prime}{\gamma^\prime}"] \arrow[d, "\map{q}{\lambda^\prime}{}"] &  W_{\gamma^\prime}  \arrow[d, "\map{q}{\gamma^\prime}{}"] \\
       % Y_{\lambda^\prime} \arrow[r, "\map{p}{\lambda^\prime}{\gamma^\prime}"]     %   &  Y_{\gamma^\prime} 
        Y_{\lambda^\prime} \arrow[r, "\map{p}{\lambda^\prime}{\gamma^\prime}"] & Y_{\gamma^\prime}
    \end{tikzcd}
\end{equation*}
from  the Proj construction, $\map{q}{\lambda^\prime}{}$ and  $\map{q}{\gamma^\prime}{}$ are torus equivariant maps with respect to the canonical torus map $T \to \frac{T}{\s}$. Also the map $\map{p}{\lambda^\prime}{\gamma^\prime}$ is a $\frac{T}{\Tsub}$-equivariant map. Now we shall demonstrate that $\Ysub$ admits canonical $ \frac{T}{\Tsub}$ action. 
To prove $\frac{T}{\Tsub}$ acts effectively on $\Ysub$, we required some evident statements which are given below.

\begin{lemma} \label{efl}
    Let $Y$ and $\Tsub$ be are two varieties with $T$ actions, the map $v : Y \to Y^\prime  $ is a $T$-equivariant birational map with $T$ acts effectively on $Y^\prime$ then $T$ acts effectively on $Y$.
    
\end{lemma}

\begin{lemma} \label{closure}
Let $Y$ be a topological space, and $\psi : Y \to Y$ be a continuous map with $A \subset X$ such that $\psi(A) \subset A$, then $\psi(\overline{A}) \subset \overline{A}$.
\end{lemma}

\begin{lemma}\label{normal}
    Let $X$ be a $T$-variety, then $\text{Norm}(X)$ has a torus action, satisfying the following commutative diagram,
\begin{equation*}
    \begin{tikzcd}
        \text{Norm}(X) \arrow[r, "\overline{t}"] \arrow[d, "g"]& \text{Norm}(X) \arrow[d,"g"] \\
        X \arrow[r, "t"] & X
    \end{tikzcd}.
\end{equation*}

%Moreover, if $\phi : Y \to X$ is a torus equivariant map with $Y$ is normal, then $\phi$ factors through $\overline{X}$ with a map $Y \to \overline{X}$ is torus equivariant.\todo{change}
\end{lemma}
%From the above commutative diagram, $g$ is a torus equivariant map.
\begin{lemma} \label{effa}
    Consider $q_1: W_{\s} \to Y_1$ the induced by the commutative diagram \ref{maindiagram}. Then, $q_1$ is a torus equivariant map, and it defines $\frac{T}{\s}$ action on $Y^\prime$.
\end{lemma}
\begin{proof}
    From Lemma \ref{closure} and \ref{normal}
\end{proof}

The map $\map{p}{\lambda^\prime}{}$ is given by,

\[  \label{arrows} \text{Norma}(\overline{q(W_{\s})}) \to \overline{q(W_{\s})} \hookrightarrow Y_1 \to Y_{\lambda^\prime} \]

and each arrow is a torus equivariant map.

\begin{proposition}
    $\y$  is a $T$-variety.
\end{proposition}
\begin{proof}
    We have to prove that $\y$ is normal variety and action of $\frac{T}{\s}$-effective. From construction, it is a normal variety and  from Lemma \ref{efl} and above arrows \ref{arrows}, the action is effective. 
\end{proof}

%\todo{Remaining part of the paper is about construction of a torus invariant  %pp-divisor   }

\section{ The proper polyhedral divisor.}\label{torusinvariantppdivisor}

Let $\Mdp$ be a character lattice associated to $\frac{T}{\Tsub}$. Consider the exact sequence \ref{exactseq}. Construction of the pp-divisor requires a homomorphism $s^\prime : \Mp \to Q(A)^*$. Note that $i(\omega)$ is a full dimensional cone, and given a $v \in i(\omega) \cap \Mp$, there is a $k \in \mathbb{N}$, such that $kv$ is saturated. For each $v \in i(\omega)$ saturated, we will define a Cartier divisor $\plD(v)$. For $v \in \text{int}(\lambdas) $ saturated, $\set{Y_{\lambdas,f}}{ f \in A_u, \text{where} \: \: i(u) = v }$ is an open cover for $Y_{\lambdas}.$ Consider the open cover $\y_f= {\map{p}{\lambda^\prime}{}}^{-1}(Y_{\lambdas,f})$.
Since $T$ acts effectively on $X$, we have a section $s: M \to Q(A)^*$ such that $s(u)$ is $u$-homogeneous. Consider the section $s^\prime : \Mp \to Q(A)^*$\label{sec} defined: \[ s^\prime(v)= s(u), \: \: \text{For fix } \: u  \in M \: \text{such that} \: i(u)=v.\]

Now consider the Cartier divisor \[\plD^\prime(v)=(\y_f,\frac{s^\prime(i(u))}{f}).\]

Since $\map{p}{\lambdas}{}$ are torus equivariant maps, so $Y_f$ are  torus invariant open subset. $\frac{s^\prime((i(u))}{f}$ is homogeneous of degree $\deg(s^\prime(i(u)))-\deg(f) \in \Mdp$ \footnote{Note that degree of element $\frac{s^\prime(i(u))}{f}$ is equal to 0 in $\Mp$ }. This defines a torus invariant pp-divisor on $\Ysub$,

\[  \plD^\prime : i(\omega) \to \text{CaDiv}_{\Q}(\Ysub), \: \; \; \: \plD^\prime(v)=\frac{1}{k} \cdot \plD^\prime(kv). \]
where $kv$ is a saturated multiple of $v$.

\bibliographystyle{amsalpha}
\bibliography{references}

\newcommand{\etalchar}[1]{$^{#1}$}
\def\cprime{$'$}
\providecommand{\bysame}{\leavevmode\hbox to3em{\hrulefill}\thinspace}
\providecommand{\MR}{\relax\ifhmode\unskip\space\fi MR }
\providecommand{\MRhref}[2]{%
  \href{http://www.ams.org/mathscinet-getitem?mr=#1}{#2}
}
\providecommand{\href}[2]{#2}
\begin{thebibliography}{KKMSD73}

\bibitem[AH06]{ah:affinetvar}
Klaus Altmann and J{\"u}rgen Hausen, \emph{Polyhedral divisors and algebraic
  torus actions}, Math. Ann. \textbf{334} (2006), no.~3, 557--607. \MR{2207875
  (2006m:14062)}

\bibitem[AHS08]{ahs:gentvar}
Klaus Altmann, J{\"u}rgen Hausen, and Hendrik S{\"u}ss, \emph{Gluing affine
  torus actions via divisorial fans}, Transform. Groups \textbf{13} (2008),
  no.~2, 215--242. \MR{2426131 (2009f:14107)}

\bibitem[AIP{\etalchar{+}}12]{aipsv:geomtvar}
Klaus Altmann, Nathan~Owen Ilten, Lars Petersen, Hendrik S\"u\ss, and Robert
  Vollmert, \emph{The geometry of {$T$}-varieties}, Contributions to algebraic
  geometry, EMS Ser. Congr. Rep., Eur. Math. Soc., Z\"urich, 2012, pp.~17--69.
  \MR{2975658}

\bibitem[BH06]{bh:gittoric}
Florian Berchtold and J{\"u}rgen Hausen, \emph{G{IT} equivalence beyond the
  ample cone}, Michigan Math. J. \textbf{54} (2006), no.~3, 483--515.
  \MR{2280492 (2008i:14072)}

\bibitem[Bri18]{MR3888690}
Michel Brion, \emph{Linearization of algebraic group actions}, Handbook of
  group actions. {V}ol. {IV}, Adv. Lect. Math. (ALM), vol.~41, Int. Press,
  Somerville, MA, 2018, pp.~291--340. \MR{3888690}

\bibitem[IV13]{iltenvollmert:upgrading}
Nathan~Owen Ilten and Robert Vollmert, \emph{Upgrading and downgrading torus
  actions}, Journal of Pure and Applied Algebra \textbf{217} (2013), no.~9,
  1583--1604.

\bibitem[KKMSD73]{keknmusd:toroidalembeddingsi}
G.~Kempf, Finn~Faye Knudsen, D.~Mumford, and B.~Saint-Donat, \emph{Toroidal
  embeddings. {I}}, Lecture Notes in Mathematics, Vol. 339, Springer-Verlag,
  Berlin-New York, 1973. \MR{0335518}

\bibitem[MFK94]{mfk:geominvth}
D.~Mumford, J.~Fogarty, and F.~Kirwan, \emph{Geometric invariant theory}, third
  ed., Ergebnisse der Mathematik und ihrer Grenzgebiete (2) [Results in
  Mathematics and Related Areas (2)], vol.~34, Springer-Verlag, Berlin, 1994.
  \MR{1304906}

\bibitem[Vol10]{vollmert:tvardvr}
Robert Vollmert, \emph{Toroidal embeddings and polyhedral divisors}, Int. J.
  Algebra \textbf{4} (2010), no.~5-8, 383--388. \MR{2652252}

\end{thebibliography}

\end{document}